\documentclass[11pt]{amsart}
\usepackage{amsfonts}
\usepackage{amscd}
\usepackage{amssymb}
\usepackage{graphics}

\usepackage{amsmath}

\usepackage{hyperref}
\hypersetup{colorlinks,citecolor=blue}

\newcommand{\BZ}{{\mathbb{Z}}}
\newcommand{\BN}{{\mathbb{N}}}

\newcommand{\BQ}{{\mathbb{Q}}}

\newcommand{\BG}{{\mathbb{G}}}

\newcommand{\gd}{\delta}

\newcommand{\gC}{\Gamma}

\newcommand{\gs}{\sigma}

\newcommand{\gO}{\Omega}

\newcommand{\gep}{\epsilon}

\newcommand{\Ad}{\text{Ad}}

\newcommand{\GL}{\text{GL}}

\newcommand{\Lie}{\text{Lie}}
\newcommand{\Gr}{\text{Gr}}

\def\sub{\text{Sub}}
\def\irs{\text{IRS}}

\newcommand{\id}[1]{Id_{#1}}

\newtheorem{prop}{Proposition}[section]
\newtheorem{thm}[prop]{Theorem}
\newtheorem{lem}[prop]{Lemma}
\newtheorem{cor}[prop]{Corollary}

\theoremstyle{definition}

\newtheorem{defn}[prop]{Definition}

\newtheorem{ques}[prop]{Question}

\newtheorem{clm}[prop]{Claim}

\newenvironment{dedication}
        {\vspace{6ex}\begin{quotation}\begin{center}\begin{em}}
        {\par\end{em}\end{center}\end{quotation}}

\catcode`\@=11
\long\def\@savemarbox#1#2{\global\setbox#1\vtop{\hsize\marginparwidth 
  \@parboxrestore\tiny\raggedright #2}}
\catcode`\@=12

\begin{document}
\author{Tsachik Gelander}

\address{Mathematics and Computer Science\\
Weizmann Institute\\
Rechovot 76100, Israel\\}
\email{tsachik.gelander@gmail.com}
\thanks{Supported by ISF-Moked grant 2095/15}

\date{\today}

\title{Kazhdan--Margulis theorem for Invarant Random Subgroups}

\maketitle

\begin{dedication}
\vspace*{.5cm}{Dedicated to David Kazhdan for his 70'th birthday\\ and to Grisha Margulis for his 70'th birthday.}
\end{dedication}

\begin{abstract}
Given a simple Lie group $G$, we show that the lattices in $G$ are weakly uniformly discrete. This is a strengthening of the Kazhdan--Margulis theorem. Our proof however is straightforward --- considering general IRS rather than lattices allows us to apply a compactness argument. In terms of p.m.p.\ actions, we show that for every $\gep>0$ there is an identity neighbourhood $U_\gep\subset G$ which intersects trivially 
the stabilizers of $1-\gep$ of the points in every non-atomic probability $G$-space.
\end{abstract}

%
%
%
%

Let $G$ be a locally compact $\gs$-compact group. We denote by $\sub(G)$ the (compact) space of closed subgroups of $G$, equipped with the Chabauty topology. 
An IRS (invariant random subgroup) of $G$ is a conjugation invariant Borel probability measure on $\sub(G)$. We denote by $\irs(G)$ 
the space of IRS's of $G$ equipped with the weak-$*$ topology. By 
Riesz' representation theorem and Alaoglu's theorem, $\irs(G)$ is compact. Every lattice $\gC\le G$ corresponds to an IRS, $\mu_\gC$, defined as the push forward of the $G$-invariant probability measure on $G/\gC$ under the map
$G/\gC\to \sub_G, ~ g\gC\mapsto g\gC g^{-1}$.

An IRS $\mu$ is called discrete if $\mu$-almost every subgroup is discrete. Recall the following variant of the Borel density theorem (see also \cite{7,GL} for similar statements):

\begin{prop}\label{prop:BD}
Let $G$ be a connected non-compact simple Lie group.
Every $\mu\in\irs(G)$ is of the form $\mu=\mu(\{ G\})\gd_G+(1-\mu(\{ G\}))\nu$ where $\nu$ is a discrete IRS.
\end{prop}

Let $\irs_d(G)$ denote the set of discrete IRS. Since $G$ (with the assumption of \ref{prop:BD})
is isolated in the space $\sub(G)$ (this follows from \cite{Za}, see \cite{G}), it follows that $\irs_d(G)$ is compact.

Let $U$ be an identity neighbourhood in $G$. A family of subgroups $\mathcal{F}\subset\sub(G)$ is called {\it $U$-uniformly discrete} if $\gC\cap U=\{1\}$ for all $\gC\in \mathcal{F}$.

\begin{defn}
A family $\mathcal{F}\subset\irs_d(G)$ of invariant random subgroups is said to be {\it weakly uniformly discrete} if for every $\gep>0$ there is an identity neighbourhood $U_\gep\subset G$ such that
$$
 \mu (\{\gC\in\sub_G:\gC\cap U_\gep\ne\{1\}\})<\gep
$$
for every $\mu\in \mathcal{F}$.
\end{defn}

\begin{thm}\label{thm:main}
Let $G$ be a connected non-compact simple Lie group. Then $\irs_d(G)$ is weakly uniformly discrete.
\end{thm}

Picking $\gep<1$ and applying the theorem for the IRS $\mu_\gC$ where $\gC\le G$ is an arbitrary lattice, one deduces the Kazhdan--Margulis theorem \cite{KM}, and in particular that there is a positive lower bound on the volume of locally $G/K$-orbifolds:

\begin{cor}[Kazhdan--Margulis theorem]\label{cor:KM}
There is an identity neighbourhood $\gO\subset G$ such that for every lattice $\gC\le G$ there is $g\in G$ such that $g\gC g^{-1}\cap \gO=\{1\}$.
\end{cor} 

Recall that every probability measure preserving $G$-space $(X,m)$ corresponds to an IRS, via the stabilizer map $X\ni x\mapsto G_x\in\sub(G)$ and that, vice versa, every IRS arises (non-uniquely) in this way (see \cite[Theorem 2.6]{7}). Thus,
Theorem \ref{thm:main} can be reformulated as follows:

\begin{thm}[p.m.p.\ actions are uniformly weakly locally free]\label{thm:p.m.p.}
For every $\gep>0$ there is an identity neighbourhood $U_\gep\subset G$ such that the stabilizers of $1-\gep$ of the points, in
any non-atomic probability measure preserving $G$-space $(X,m)$ are $U_\gep$-uniformly discrete. I.e.,
there is a subset $Y\subset X$ with $m(Y)>1-\gep$ such that $U_\gep\cap G_y=\{1\},~\forall y\in Y$.
\end{thm}

 A famous conjecture of Margulis \cite[page 322]{Ma} asserts that the set of all torsion-free anisotropic arithmetic lattices in $G$ is $U$-uniformly discrete for some identity neighbourhood $U\subset G$. 
Theorem \ref{thm:main} can be regarded as a probabilistic variant of this conjecture as it implies that all lattices in $G$ are jointly weakly uniformly discrete.

%
%

\section{Proofs}\label{sec:proofs}

For the sake of completeness, let us start by:

\begin{proof}[Sketching a proof for Proposition \ref{prop:BD}]
By the Von-Neumann--Cartan theorem, every closed subgroup of $G$ is a Lie group. Consider the map $\psi:\sub_G\to\Gr(\Lie(G))$, $H\mapsto \Lie(H)$. It is upper semi-continuous, hence measurable. Let $\nu=\psi_*(\mu)$. Then $\nu$ is $\Ad(G)$-invariant and hence, by the Furstenberg lemma \cite{Fu}, is supported on $\{\Lie(G),\{0\}\}$. Finally observe that $\psi^{-1}(\Lie(G))=G$ and $\psi^{-1}(\{0\})$ is the set of discrete subgroups of $G$.
\end{proof}



Let $U_n, n\in\BN$ be a descending sequence of compact sets in $G$ which form a base of identity neighbourhoods, and set
$$
  K_n=\{\gC\in\sub_G:\gC\cap U_n=\{\id G \}\}.
$$

Since $G$ has NSS (no small subgroups), i.e. there is an identity neighbourhood which contains no non-trivial subgroups, we have:

\begin{lem}\label{K_n-open}
The sets $K_n$ are open in $\sub(G)$.
\end{lem}

\begin{proof}
Fix $n$ and
let $V\subset U_n$ be an open identity neighbourhood which contains no non-trivial subgroups, such that ${V^2}\subset U_n$. It follows that a subgroup $\gC$ intersects $U_n$ non-trivially iff it intersects $U_n\setminus V$. Since $U_n\setminus V$ is compact, the lemma is proved.  
\end{proof}

In addition, observe that the ascending union $\bigcup_n K_n$ exhausts $\sub_d(G)$, the set of all discrete subgroups of $G$. Therefore we have:

\begin{clm}\label{clm}
For every $\mu\in\irs_d(G)$ and $\gep>0$ we have $\mu(K_n)>1-\gep$ for some $n$.\qed
\end{clm}

Let
$$
 \mathcal{K}_{n,\gep}:= \{\mu\in\irs_d(G):\mu(K_n)>1-\gep\}.
$$
Since $\sub(G)$ is metrizable, it follows from Lemma \ref{K_n-open} that $\mathcal{K}_{n,\gep}$ is open. 
By Claim \ref{clm}, for any given $\gep>0$, the sets $\mathcal{K}_{n,\gep},~n\in\BN$ form an ascending cover of $\irs_d(G)$. Since the latter is compact, we have $\irs_d(G)\subset \mathcal{K}_{m,\gep}$ for some $m=m(\gep)$. 
It follows that 
$$
 \mu\big(\{ \gC\in\sub(G):\gC~\text{intersects}~U_m~\text{trivially}\}\big)>1-\gep,
$$
for every $\mu\in\irs_d(G)$. This completes the proof of Theorem \ref{thm:main}.
\qed

\section{The semisimple case}

Theorem \ref{thm:main} holds more generally for semisimple groups:

\begin{thm}
Let $G$ be a connected center-free semisimple Lie group with no compact factors. Then $\irs_d(G)$ is weakly uniformly discrete. 
\end{thm}

The additional ingredient, needed in order to apply the proof above is:

\begin{prop}
$\irs_d(G)$ is compact.
\end{prop}

\begin{proof}
Let ${\mathcal{D}}=\overline{\sub_d(G)}\subset\sub(G)$ denote the closure of the set of discrete subgroups of $G$. As follows from \cite{Za} (see also \cite[Theorem 4.1.7]{Thurston}), the Lie algebra of any $H\in {\mathcal{D}}$ is nilpotent. In particular, every $H\in {\mathcal{D}}$ with semisimple Zariski closure is discrete. On the other hand, by the IRS version of the Borel density theorem (see \cite{GL} for the case where $G$ is not simple) for every IRS $\mu$, the Zariski closure of $\mu$-almost every subgroup is connected and semisimple. It follows that $\irs_d(G)$ coincides with $\text{Prob}({\mathcal{D}})^G$, the space of $G$-invariant probability measures on the compact space ${\mathcal{D}}$. Hence, by Alaoglu's theorem it is compact.
\end{proof}

\section{General groups}
Consider now a general locally compact $\gs$-compact group $G$.
Since $\sub_d(G)\subset\sub(G)$ is a measurable subset, by restricting attention to it, one may replace Property NSS
by the weaker Property NDSS (no discrete small subgroups), which means that there is an identity neighbourhood which contains no non-trivial discrete subgroups. In that generality, the analog of Lemma \ref{K_n-open} would say that $K_n$ are relatively open in $\sub_d(G)$. Thus, the ingredients required for the argument above are:


\begin{enumerate}
\item $\irs_d(G)$ is compact,
\item $G$ has NDSS.
\end{enumerate}

In particular we have:

\begin{thm}\label{thm:general}
Let $G$ be a locally compact $\gs$-compact group which satisfies $(1)$ and $(2)$. Then $\irs_d(G)$ is weakly uniformly discrete.
\end{thm}

As explained in Section \ref{sec:proofs}, if $G$ possesses the Borel density theorem (i.e. the analog of Proposition \ref{prop:BD}) and $G$ is isolated in $\sub(G)$ then $(1)$ holds. By \cite{GL}, simply-connected simple groups over non-Archimedean local fields enjoy these two properties, hence satisfy $(1)$. 

For a locally compact totally disconnected group $G$, the following are equivalent:
\begin{itemize}
\item $G$ has property NDSS.
\item $G$ admits an open torsion-free subgroup.
\item The space $\sub_d(G)$ is uniformly discrete.
\end{itemize}

\paragraph{\bf $p$-adic groups.}
Note that a $p$-adic analytic group $G$ has NDSS,
and hence $\irs_d(G)$ is uniformly discrete (in the obvious sense).
%
Moreover, if $G\le\GL_n(\BQ_p)$ is a rational algebraic subgroup, then the first principal congruence subgroup $G(p\BZ_p)$ is a torsion-free open compact subgroup. Supposing further that $G$ is simple, then in view of the Borel density theorem proved in \cite{GL}, every IRS of $G$ is supported on $\{ G\}\cup\sub_d(G)$. Therefore we have:

\begin{prop}[See \cite{GL} for more details]
Let $G\le \GL_n(\BQ_p)$ be a simple $\BQ$-algebraic group. Let $(X,\mu)$ be a probability $G$-space essentially with no global fixed points.
Then the action of the congruence subgroup $G(p\BZ_p)$ on $X$ is essentially free.
\end{prop} 

\paragraph{\bf Positive characteristic.}
Algebraic groups over local fields of positive characteristic do not posses property NDSS, and our argument does not apply to them.

\begin{ques}
Let $k$ be a local field of positive characteristic, let $\BG$ be simply connected absolutely almost simple $k$-group with positive $k$-rank 
and let $G=\BG(k)$ be the group of $k$-rational points. Is $\irs_d(G)$ weakly uniformly discrete? 
\end{ques}

\noindent
The analog of Corollary \ref{cor:KM} in positive characteristic was proved in \cite{Ali,Rag}.

\end{document}